\documentclass[10pt]{article}
\usepackage{graphicx}
\usepackage{tikz}
\usepackage{amsfonts,amsmath,amssymb,amsthm}
\usepackage{fixmath}
\usepackage{hyperref}
\usepackage{cite}

\begin{document}

\newtheorem{theorem}{Theorem}[section]
\newtheorem{corollary}[theorem]{Corollary}
\newtheorem{definition}[theorem]{Definition}
\newtheorem{conjecture}[theorem]{Conjecture}
\newtheorem{question}[theorem]{Question}
\newtheorem{lemma}[theorem]{Lemma}
\newtheorem{proposition}[theorem]{Proposition}
\newtheorem{claim}[theorem]{Claim}
\newtheorem{example}[theorem]{Example}
\newtheorem{problem}[theorem]{Problem}

\title{ Maximum atom-bond connectivity index with given graph parameters \thanks{ This work is supported by the National Natural Science Foundation of China (Nos.11531001 and 11271256),  the Joint NSFC-ISF Research Program (jointly funded by the National Natural Science Foundation of China and the Israel Science Foundation(No. 11561141001)),  Innovation Program of Shanghai Municipal Education Commission (No. 14ZZ016, No. 15ZZ108), and Simons Foundation (No. 245307).\newline \indent
 $^{\dagger}$Corresponding author:  Xiao-Dong Zhang (Email: xiaodong@sjtu.edu.cn)}}
\author{Xiu-Mei Zhang$^{1}$, Yu Yang$^2$,  Hua  Wang$^{3}$ and  Xiao-Dong Zhang$^{4\dagger}$\\
{\small $^{1}$Department of Mathematics, Shanghai Sanda University}\\
{\small  2727 Jinhai road, Shanghai, 201209, P. R. China}\\
{\small $^{2}$School of Information Science and Technology, Dalian Maritime University}\\
{\small Dalian, 116026, P. R. China}\\
{\small $^{3}$Department of Mathematical Sciences, Georgia Southern
University} \\
{\small Stateboro, GA 30460,  United States}\\
{\small $^{4\dagger}$ Department of Mathematics,
 Shanghai Jiao Tong University} \\
{\small  800 Dongchuan road, Shanghai, 200240, P. R. China}
}

\date{}
\maketitle
 \begin{abstract}
  The atom-bond connectivity (ABC) index is a degree-based topological index. It was introduced due to its applications in modeling the properties of certain molecular structures and has been since extensively studied. In this note, we examine the influence on the extremal values of the ABC index by various graph parameters. More specifically, we consider the maximum ABC index of connected graphs of given order, with fixed independence number, number of pendent vertices, chromatic number and edge-connectivity respectively. We provide characterizations of extremal structures as well as some conjectures. Numerical analysis of the extremal values are also presented.

 \end{abstract}


 \vskip 0.5cm

\section{Introduction and preliminaries}

Let $G$ be a simple graph with vertex set $V(G)$ and edge set $E(G)$. For $u \in V(G)$, the degree of $u$, denoted by $d(u)$, is the
number of neighbors of $u$ in $G$. An independent set is a set of vertices of which no pair is adjacent.
The independence number $\beta(G)$ of a graph $G$ is the size of a largest independent set of $G$. The chromatic number $\chi(G)$ of a graph $G$ is the least number of colors assigned to $V(G)$ such that no adjacent elements receive the same color.
The edge connectivity $k(G)$ of a graph $G$ is the minimum number of edges needed to disconnect $G$.

The atom bond connectivity (ABC) index of $G$ is defined \cite{estrada1998} as
$$ABC(G) = \sum_{uv\in E(G)}\sqrt{\frac{d(u)+d(v)-2}{d(u)d(v)}}.$$
The ABC index is one of many so called topological indices that are extensively used in theoretical chemistry to correlate
physico-chemical properties with the molecular structures of chemical compounds.
It appears that the ABC index shows a strong correlation with heat of formation of alkanes \cite{estrada1998}.  Some topological approaches were also developed basing on the ABC index to explain the differences in the energy of linear and branched alkanes \cite{estrada2008}.

In the study of topological indices in general, it is often of interest to consider the extremal values of a certain index among graphs under various constrains. Along this line, the extremal values of the ABC index have been extensively explored \cite{chen2012,furtula2009,gan2012,xing2011,dimitrov2014,xing2010,xing2012,das2012,zhou2011,mGoubko2015,cmagnant2015,ddimitrov201601,ddimitrov201602}.

We intend to expand this study by exploring the maximum ABC index of connected graphs of given order, with fixed independence number, number of pendent vertices, edge-connectivity, and chromatic number respectively. First we will introduce some simple but useful facts.

\begin{theorem}[\cite{chen2011}]\label{theorem2-2}
Let $G$ be a graph with n vertices, if $x,y\in V(G)$ and $xy\in E(\overline{G})$, then
$$ ABC(G)\leqslant ABC(G+xy) $$
with equality if and only if $x$ and $y$ are both isolated vertices. Furthermore,
$$ ABC(G)\leqslant ABC(K_n) $$
with equality if and only if $G = K_n$.
\end{theorem}

To simplify notations, we define the following functions:
\begin{itemize}
\item $f(x,y)= \sqrt{\frac{x+y-2}{xy}}$;
\item $g(x,y)=f(x+1,y)-f(x,y)$;
\item $F(x)=xf(x+m,1)$,
\end{itemize}
for $x,y,m\geq 1$.
\begin{lemma}[\cite{xing2011}]\label{theorem2-3}
For the function $f(x,y)$ we have:
\begin{itemize}
\item $f(x,1)$ is strictly increasing with respect to $x$;
\item $f(x,2)=\frac{\sqrt{2}}{2}$;
\item $f(x,y)$ is strictly decreasing with respect to $x$ for any fixed $y\geq 3$.
\end{itemize}

\end{lemma}

\begin{lemma}[\cite{dimitrov2014,xing2011}]\label{theorem2-4}
The function $g(x,y)$ is strictly decreasing with respect to $x$ if $y=1$, and increasing with respect to $x$ if $y\geq 2$.
\end{lemma}

\begin{lemma}\label{theorem2-5}
The function $F(x)$ is convex and strictly increasing for $x\geq 1$. As a result of the convexity we have
$$F(x_1+1)-F(x_1)>F(x_2)-F(x_2-1) $$
if $x_1\geq x_2\geq 1$.
\end{lemma}
\begin{proof}
Note that $F(x)=xf(x+m,1)=x\sqrt{\frac{x+m-1}{x+m}}$ , then we have
\begin{align*}
F'(x) & = \sqrt{\frac{x+m-1}{x+m}}+\frac{x}{2}\cdot\sqrt{\frac{x+m}{x+m-1}}\cdot \frac{1}{(x+m)^2}  \\
 & = \sqrt{\frac{x+m-1}{x+m}}\left(1+\frac{1}{2}\cdot \frac{x}{(x+m-1)(x+m)}\right)>0
\end{align*}
and
\begin{align*}
F''(x) & = \frac{1}{2}\cdot\sqrt{\frac{x+m}{x+m-1}}\cdot\frac{1}{(x+m)^2}\left(1+\frac{1}{2}\left(\frac{m}{x+m}-\frac{m-1}{x+m-1}\right)\right) \\
&  \quad\quad\quad\quad\quad\quad +\sqrt{\frac{x+m-1}{x+m}}\left(-\frac{m}{2(x+m)^2}+\frac{m-1}{2(x+m-1)^2}\right) \\
 & = \sqrt{\frac{x+m-1}{x+m}}\left(\frac{1+\frac{1}{2}(\frac{m}{x+m}-\frac{m-1}{x+m-1})}{2(x+m-1)(x+m)}-\frac{m}{2(x+m)^2}+\frac{m-1}{2(x+m-1)^2}\right) \\
 & = \frac{(4m-1)x+4m(m-1)}{4(x+m)^{\frac{5}{2}}(x+m-1)^{\frac{3}{2}}}>0
\end{align*}
when $x,m\geq 1$.
\end{proof}

\begin{lemma}\label{theorem2-6}
Let $G(a,b)= f(a,b-1)-f(a-1,b)$ for some $a>b>0$. Then $G(a,b)>0$.
\end{lemma}
\begin{proof}
This follows from direct calculations.
\end{proof}

In the following sections we will first explore the maximum ABC index of graphs of given order and various fixed parameters. Based on these results some computational analysis is provided. In the end we briefly discuss some other questions and pose a couple of conjectures.

\section{Maximum ABC index with given independence number or number of pendent vertices}

In this section we characterize the extremal graph on $n$ vertices, with given independence number (Theorem~\ref{theorem4-1}) and with given number of pendent vertices (Theorem~\ref{theorem5-1}).

\begin{definition}\label{def:cup}
For two vertex-disjoint graphs $G$ and $H$, the $join$ of $G$ and $H$, denoted by $G\vee H$, is a graph with vertex set
$V(G)\cup V(H)$ and edge set $E(G)\cup E(H)\cup \{xy\mid x\in V(G), y\in V(H)\}$.
\end{definition}

\begin{theorem}\label{theorem4-1}
 Let $G$ be a connected graph on $n$ vertices and independence number $\beta$. Then
 $$ABC(G)\leq \beta(n-\beta)\sqrt{\frac{2n-\beta-3}{(n-\beta)(n-1)}}+\frac{(n-\beta)(n-\beta-1)}{2}\sqrt{\frac{2n-4}{(n-1)(n-1)}}$$
 with equality if and only if $G \cong \overline{K_{\beta}}\bigvee K_{n-\beta}$.
\end{theorem}
\begin{proof}
Suppose $G^\ast$ is the graph with the maximum $ABC$ index among all $n-$vetex connected graphs with independence number $\beta$.

Let $S$ be a maximal independent set in $G^\ast$ with $|S| = \beta$. By Theorem~\ref{theorem2-2}, adding edges to a graph will increase its $ABC$ index. Thus each vertex $x$ in $S$ is adjacent to every vertex $y$ in $G^\ast-S$ and the subgraph induced by vertices in $G^\ast-S$
is $K_{n-\beta}$.
Consequently $G^\ast \cong \overline{K_{\beta}}\bigvee K_{n-\beta}$. Direct calculations yield
\begin{align*}
 & ABC(\overline{K_{\beta}}\bigvee K_{n-\beta}) \\
= & \beta(n-\beta)\sqrt{\frac{2n-\beta-3}{(n-\beta)(n-1)}}+\frac{(n-\beta)(n-\beta-1)}{2}\sqrt{\frac{2n-4}{(n-1)(n-1)}}.
\end{align*}
\end{proof}

\begin{definition}
For convenience we employ the following notations:
\begin{itemize}
\item let $K_n^p$ denote the graph obtained from attaching $p$ pendent edges to one vertex of $K_{n-p}$; and
\item let $G'$ denote the graph obtained from attaching $n-3$  pendent edges to one end of a path on three vertices.
\end{itemize}
\end{definition}

\begin{theorem}\label{theorem5-1}
Let $G$ be a connected graph on $n$ vertices with $p$ pendent vertices, then:
\begin{enumerate}
\item If $n-p=1$, then $ABC(G)=ABC(S_n)=\sqrt{(n-1)(n-2)}$;
\item If  $n-p=2$, then $ABC(G)\leq (n-3)\sqrt{\frac{n-3}{n-2}}+\sqrt{2}$
 with equality if and only if $G \cong  G'$;
\item If  $n-p>2$, then
\begin{align*}
 ABC(G) & \leq p\sqrt{\frac{n-2}{n-1}}+(n-p-1)\sqrt{\frac{2n-p-4}{(n-1)(n-p-1)}} \\
& \quad\quad\quad\quad\quad\quad\quad\quad +\frac{(n-p-1)(n-p-2)}{2}\sqrt{\frac{2n-2p-4}{(n-p-1)^{2}}}
\end{align*}
 with equality if and only if $G \cong K_n^p $.
\end{enumerate}
\end{theorem}

\begin{proof}
Let $G^\ast$ be the graph with the maximum $ABC$ index  among all $n-$vetex connected graphs with $p$ pendent vertices.

\noindent {\bf Case 1:} If $n-p=1$, then $G^\ast$ is the star.

\noindent {\bf Case 2:} If $n-p=2$, then $G^\ast$ is the graph obtained by attaching $a_1$ pendent edges to one vertex $v_1$
and $a_2(=p-a_1)$ pendent edges to the other vertex $v_2$ of $K_{2}$.

Assuming, without loss of generality, that $a_1\geq a_2\geq 1$, we claim that $a_1=p-1=n-3$ and $a_2=1$ (note that in this case $G^\ast \cong G'$).

Otherwise, if $a_2\geq 2$, let $G_1$ by obtained from $G^\ast$ by detaching and reattaching one of the pendent edges from $v_2$ to $v_1$. Then
\begin{align*}
ABC(G_1)-ABC(G^\ast) & =((a_1+1)f((a_1+1)+1,1)-a_1f(a_1+1,1)) \\
& \quad\quad\quad -(a_2f(a_2+1,1)-(a_2-1)f((a_2-1)+1,1)) \\
& \quad\quad\quad\quad\quad +(f(a_1+2,a_2)-f(a_1+1,a_2+1)).
\end{align*}
Let $x_1=a_1$, $x_2=a_2$ and $m=1$ in Lemma~\ref{theorem2-5}, we have
$$F(a_1+1,1)-F(a_1,1)>F(a_2,1)-F(a_2-1,1). $$
Or equivalently,
\begin{align*}
& (a_1+1)f((a_1+1)+1,1)-a_1f(a_1+1,1)) \\
& \quad\quad\quad\quad\quad -(a_2f(a_2+1,1)-(a_2-1)f((a_2-1)+1,1)>0.
\end{align*}
Applying Lemma~\ref{theorem2-6} with $a=a_1+2$ and $b=a_2+1$ yields $f(a_1+2,a_2)-f(a_1+1,a_2+1)>0$.

Consequently $ABC(G_1)-ABC(G^\ast)>0$, a contradiction. The conclusion then follows from direct calculations.

\noindent {\bf Case 3:} If $n-p>2$, let $P$ be the set of pendent vertices in $G^\ast$ with $|P| = p$. Again by Theorem~\ref{theorem2-2}, the subgraph induced by vertices in $G^\ast-P$ must be $K_{n-p}$. Label the vertices of this $K_{n-p}$
  as $v_1,v_2,\cdots v_{n-p}$ and let the number of pendent vertices adjacent to each vertex $v_i$ be $a_i$ with
$a_1\geq a_2\geq \cdots \geq a_{n-p}\geq 0$.

If $a_1=p$ and $a_2=\cdots =a_{n-p}= 0$, then $G^\ast\cong K_n^p$.

If $G^\ast\ncong K_n^p$, then $a_1\geq a_2\geq 1$. Consider $G_2$ obtained from detaching one of the pendent edges of $v_2$ and reattaching  to $v_1$. We have
\begin{align*}
  & ABC(G_2)-ABC(G^\ast) \\
    = & \sum_{i=3}^{n-p}[(f(a_1+n-p,a_i+n-p-1)-f(a_1+n-p-1,a_i+n-p-1))\\
    & \quad\quad -(f(a_2+n-p-1,a_i+n-p-1)-f(a_2+n-p-2,a_i+n-p-1))]\\
    & \quad +(f(a_1+n-p,a_2+n-p-2)-f(a_1+n-p-1,a_2+n-p-1)) \\
    & \quad +((a_1+1)f((a_1+1)+n-p-1,1)-a_1f(a_1+n-p-1,1)) \\
    & \quad -(a_2f(a_2+n-p-1,1)-(a_2-1)f((a_2-1)+n-p-1,1)).
    \end{align*}
 For each $i=3,4,\cdots,n-p$, we have $a_i+n-p-1\geq 2$ , $a_1+n-p-2\geq 2$ and $a_1+n-p-1>a_2+n-p-2$. Then
 by Lemma \ref{theorem2-4},
\begin{align*}
 & (f(a_1+n-p,a_i+n-p-1)-f(a_1+n-p-1,a_i+n-p-1)) \\
& \quad\quad\quad -(f(a_2+n-p-1,a_i+n-p-1)-f(a_2+n-p-2,a_i+n-p-1)) \\
& \quad\quad\quad\quad\quad\quad \geq 0.
\end{align*}
From $a_1+n-p>a_2+n-p-1$ and Lemma \ref{theorem2-6}, we have
\begin{align*}
& G(a_1+n-p,a_2+n-p-1) \\
= & f(a_1+n-p,a_2+n-p-2)-f(a_1+n-p-1,a_2+n-p-1)>0.
\end{align*}
Let $m=n-p-1$, by Lemma \ref{theorem2-5}, we have
$$F(a_1+1)-F(a_1)>F(a_2)-F(a_2-1) $$
for $a_1\geq a_2$.

As a consequence we have
\begin{align*}
& ((a_1+1)f((a_1+1)+n-p-1,1)-a_1f(a_1+n-p-1,1)) \\
& \quad\quad\quad -(a_2f(a_2+n-p-1,1)-(a_2-1)f((a_2-1)+n-p-1,1))>0
\end{align*}
and hence $ABC(G_2)-ABC(G^\ast)>0$, a contradiction. Thus $G \cong K_n^p $ and the conclusion follows.
\end{proof}

\section{Maximum ABC index with given edge-connectivity}

In this section we consider the maximum ABC index of graphs of given order and edge-connectivity. The conclusion is, to some extent, expected. But the proof turned out to be rather complicated.

\begin{theorem}\label{theorem6-1}
 Let $G$ be a connected graph on $n\geq 6$ vertices and edge-connectivity $k \geq 2$. Then
 $$ABC(G)\leq k\sqrt{\frac{n+k-3}{k(n-1)}}+\frac{k(k-1)}{2(n-1)}\sqrt{2n-4}+\frac{(n-k-1)(n-k-2)}{2(n-2)}\sqrt{2n-6}$$\\
 $$+k(n-k-1)\sqrt{\frac{2n-5}{(n-1)(n-2)}}$$
 with equality if and only if $G \cong K_k\vee (K_1+K_{n-k-1}) $.
\end{theorem}

Note that $K_k\vee (K_1+K_{n-k-1})$ is simply the graph obtained from joining one vertex with $k$ of the vertices in $K_{n-1}$.

\begin{proof} Suppose $G^\ast$ is the graph with maximum $ABC$ index among all graphs of order $n\geq 6$ and edge-connectivity $k\geq 2$,
let ${e_1,e_2,\cdots,e_k}$ be a $k$-edge cut in $G^\ast$ and let $G_1,G_2$ be the connected components in $G^\ast-\{e_1,e_2,\cdots,e_k\}$. Again by Theorem \ref{theorem2-2}, both $G_1$ and $G_2$ must be complete graphs. Let $n_i$ be the number of vertices of $G_i$ ($i=1,2)$, then $n_1 + n_2 = n$.

Without loss of generality, let $n_2\geq n_1$. If $n_1=1$, then $G^\ast \cong K_k\vee (K_1+K_{n-k-1})$.

Now we focus on the case of $n_2\geq n_1\geq 2$. For $i=1,2$, $G_i$ has $\frac{n_i(n_{i}-1)}{2}$ edges, for $G_i$ is a complete graph. On the other hand, the sum of degrees of all vertices in $G_i$ is at least $n_ik$, for the minimum degree of $G^\ast$ is at least $k$. Thus $G_i$ has at least $\frac{n_ik-k}{2}$ edges. Hence $\frac{n_i(n_i-1)}{2}\geq \frac{n_ik-k}{2} = \frac{k(n_i-1)}{2}$, implying that $n_i\geq k$. Consequently we can assume $n_2\geq n_1\geq k$.

Firstly, if there is  a vertex, say $v$ in $V(G^\ast)$, of degree $k$. Let $v_1,\cdots v_k$ be the neighbors of $v$. Write $A=\{v_1,\cdots v_k\}$ and $B=V(G^\ast)\backslash \{v,v_1,\cdots v_k\}$. If $G[A\cup B]$, the subgraph of $G^\ast$ induced by $V(A\cup B)$, is the complete graph $K_{n-1}$, then
 $G^\ast \cong K_k\vee (K_1+K_{n-k-1}) $ as claimed.

Otherwise, there exsit $x,y\in V(G^\ast)$ such $xy\in E(\overline{G[A\cup B]})$. But then the graph $G'=G^\ast + xy$ still have edge-connectivity $k$ with $ABC(G')>ABC(G^\ast)$ by Theorem \ref{theorem2-2}, contradiction.

If, on the other hand, $d_{G^\ast}(v)\geq k+1$ for every vertex $v\in V(G^\ast)$. Then we must have $n_2\geq n_1\geq k+1$ by similar arguments. We now show that the maximum $ABC$ index cannot be achieved in this case.

If $n_2\geq n_1\geq k+1\geq 3$, by  Lemma \ref{theorem2-3} we have
\begin{align*}
ABC(G^\ast) & < \frac{n_1(n_1-1)}{2}f(n_1-1,n_1-1) \\
& \quad\quad\quad\quad\quad\quad +\frac{n_2(n_2-1)}{2}f(n_2-1,n_2-1)+kf(n_1,n_2) \\
& <\frac{n_1^\frac{3}{2}}{\sqrt{2}}+\frac{n_2^\frac{3}{2}}{\sqrt{2}}+k\sqrt{\frac{n_1+n_2-2}{n_1n_2}}.
\end{align*}

Note that when $n_1=1$, also by Lemma \ref{theorem2-3} we have
\begin{align*}
ABC(G^\ast) & > kf(k,n-1)+\frac{(n-1)(n-2)}{2}f(n-1,n-1) \\
 & = k\sqrt{\frac{n+k-3}{k(n-1)}}+\frac{(n-1)(n-2)}{2}\sqrt{\frac{2(n-2)}{(n-1)(n-1)}}\\
& =k\sqrt{\frac{n+k-3}{k(n-1)}}+\frac{(n-2)^\frac{3}{2}}{\sqrt{2}}.
\end{align*}

We now make use of the following fact.

\begin{claim}\label{claim}
For $n \geq 10$ and $3\leq k+1\leq n_1\leq \frac{n}{2}$,
\begin{equation}\label{eq:claim}
k\sqrt{\frac{n+k-3}{k(n-1)}}+\frac{(n-2)^\frac{3}{2}}{\sqrt{2}}>\frac{n_1^\frac{3}{2}}{\sqrt{2}}+
\frac{n_2^\frac{3}{2}}{\sqrt{2}}+k\sqrt{\frac{n_1+n_2-2}{n_1n_2}}.
\end{equation}
\end{claim}

Note that our conclusion follows from \eqref{eq:claim}. For $6\leq n\leq 9$, it is easy to check that the case $n_1=1$ yields larger $ABC$ index than  the case $n_1\geq k+1$, for each $2\leq k\leq \frac{n}{2}-1$.
\end{proof}

In the rest of this section we provide a proof to \eqref{eq:claim}.

\begin{proof}[Proof of Claim~\ref{claim}]

First note that \eqref{eq:claim} is equivalent to
$$ (n-2)^\frac{3}{2}-(n_1^\frac{3}{2}+(n-n_1)^\frac{3}{2})>\sqrt{2}k\left(\sqrt{\frac{n-2}{n_1(n-n_1)}}-\sqrt{\frac{n+k-3}{k(n-1)}}\right) . $$
For $3\leq k+1\leq n_1\leq \frac{n}{2}$, let
$$h(n,n_1)=(n-2)^\frac{3}{2}-(n_1^\frac{3}{2}+(n-n_1)^\frac{3}{2})$$
and
$$l(n,k, n_1)=\sqrt{2}k\left(\sqrt{\frac{n-2}{n_1(n-n_1)}}-\sqrt{\frac{n+k-3}{k(n-1)}}\right).$$
Then $h(n,n_1)$ is strictly increasing and $l(n,k, n_1)$ is strictly decreasing for $3\leq k+1\leq n_1\leq \frac{n}{2}$. Hence
$$ h(n,n_1)\geq h(n,k+1)\geq h(n,3)=(n-2)^\frac{3}{2}-(3^\frac{3}{2}+(n-3)^\frac{3}{2}) $$
and $l(n,k, n_1)\leq l(n,k, k+1)$.

We now show that
$$l(n,k, k+1)=\sqrt{2}k\left(\sqrt{\frac{n-2}{(k+1)(n-k-1)}}-\sqrt{\frac{n+k-3}{k(n-1)}}\right)$$
is increasing for $2\leq k\leq \frac{n}{2}-1$ and $n\geq 20$.

It is easy to obtain the formula (which we skip for it is too long and not informative) of $l'_k(n,k, k+1)$, and see (with the help of computer) that $l'_k(n,k, k+1)$ is positive for some particular values of $k$ and $n$. We may claim that $l'_k(n,k, k+1)$ is positive for $2\leq k\leq \frac{n}{2}-1$ and $n\geq 20$ by showing $l'_k(n,k, k+1)=0$ is not possible (and hence $l'_k(n,k, k+1)$ must be always positive).

Thanks to computer algebra, we have that $l'_k(n,k, k+1)=0$ is equivalent to
 \begin{align*}
 0 = & (k^{2}+k-1)n^{5}-(8k^{2}+6k-9)n^{4}+(5k^{5}+8k^{4}+19k^{2}+6k-30)n^{3}\\
 & \quad +(k^{6}-24k^{5}-42k^{4}-2k^{3}-3k^{2}+20k+46)n^{2}\\
 & \quad -(8k^{7}+6k^{6}-57k^{5}-75k^{4}+10k^{3}+36k^{2}+39k+33)n\\
 & \quad +4k^{8}+12k^{7}-3k^{6}-46k^{5}-37k^{4}+16k^{3}+27k^{2}+18k+9,
\end{align*}
or equivalently
\begin{align*}
0 = & ((k^{2}+k-1)n^{5}-(8k^{2}+6k-9)n^{4}) \\
& \quad +[(k^{5}+8k^{4}+19k^{2}+6k-30)n^{3} \\
 & \quad\quad\quad\quad -(24k^{5}+42k^{4}+2k^{3}+3k^{2}-20k-46)n^{2}] \\
& \quad +[(57k^{5}+75k^{4}-10k^{3}-36k^{2}-39k-33)n \\
& \quad\quad\quad\quad -(46k^{5}+37k^{4}-16k^{3}-27k^{2}-18k-9)] \\
& \quad +(4k^{5}n^{3}+k^{6}n^{2}-(8k^{7}+6k^{6})n)+4k^{8}+12k^{7}-3k^{6} .
\end{align*}
For simplicity we denote the above expression by $H(n,k)$. It is then straightforward to check the followings:
\begin{itemize}
\item $(k^{2}+k-1)n^{5}-(8k^{2}+6k-9)n^{4}>0$ when $n\geq 8$;
\item  $(k^{5}+8k^{4}+19k^{2}+6k-30)n^{3}
 -(24k^{5}+42k^{4}+2k^{3}+3k^{2}-20k-46)n^{2}>0$ when $n\geq 24$;
\item  for $2\leq k< \frac{n}{2}$, we have $4k^{8}+12k^{7}-3k^{6}>0$ and
  $4k^{5}n^{3}+k^{6}n^{2}-(8k^{7}+6k^{6})n>0$;
\item for $20\leq n\leq 23$, simple calculation shows $H(n,k)>0$.
\end{itemize}
Thus, $l(n,k, k+1)$ is increasing when  $n\geq 20$ and $2\leq k\leq \frac{n}{2}-1$.

Consequently
$$l(n,k, k+1)\leq l(n,\frac{n}{2}-1,\frac{n}{2}) $$
when $n$ is even and
$$l(n,k, n_1)\leq l(n,k, k+1)\leq l(n,\frac{n-1}{2}-1,\frac{n-1}{2})$$
when $n$ is odd. We now discuss different cases to finish the proof:
\begin{itemize}
\item If $10\leq n\leq 13$, the case $n_1=1$ yields larger $ABC$ index than  the case $n_1\geq k+1=3$ and we always have $h(n,k+1) \geq l(n,k,k+1)$ for $3 \leq k \leq \frac{n}{2}-1$;
\item If $14\leq n\leq 19$, we always have $h(n,k+1)\geq l(n,k, k+1)$ for $2\leq k\leq \frac{n}{2}-1$;
\item If $20\leq n\leq 48$, we always have $h(n,3)\geq l(n,\frac{n}{2}-1,\frac{n}{2})$ when $n$ is even and
  $h(n,3)\geq l(n,\frac{n-1}{2}-1,\frac{n-1}{2})$ when $n$ is odd;
\item If $n\geq 49$, we have
\begin{align*}
 h(n,3) & =(n-2)^\frac{3}{2}-(3^\frac{3}{2}+(n-3)^\frac{3}{2}) \\
& =\frac{3n^{2}-15n+19}{\sqrt{n^{3}-6n^{2}+12n-8}+\sqrt{n^{3}-9n^{2}+27n-27}}-3^\frac{3}{2} \\
& >\frac{3n^{2}-15n+19}{\sqrt{n^3}+\sqrt{n^3}}-3^\frac{3}{2} \\
& >\frac{3}{2}(n^\frac{1}{2}-5n^{-\frac{1}{2}})-3^\frac{3}{2} .
\end{align*}
On the other hand,
\begin{align*}
 & l\left(n,\frac{n}{2}-1,\frac{n}{2}\right) \\
= & \sqrt{2}\left(\frac{n}{2}-1\right)\left(\sqrt{\frac{n-2}{\frac{n}{2}(n-\frac{n}{2})}}-\sqrt{\frac{n+\frac{n}{2}-1-3}{(\frac{n}{2}-1)(n-1)}}\right) \\
 < & \frac{\sqrt{2n}}{2}\frac{n^{3}-12n^{2}+32n-16}{2\sqrt{n^{6}-6n^{5}+13n^{4}-12n^{3}+4n^{2}}+\sqrt{3n^{6}-11n^{5}+8n^{4}}}\\
< & \frac{\sqrt{2n}}{2}\frac{1+32n^{-2}}{2\sqrt{1-6n^{-1}}+\sqrt{3-11n^{-1}}} \\
\leq & \frac{\sqrt{2n}}{2}\frac{1+32n^{-2}}{2\sqrt{1-6\times 49^{-1}}+\sqrt{3-11\times 49^{-1}}} \\
= & \frac{\sqrt{2n}}{2} \frac{7}{2\sqrt{43}+\sqrt{136}}\left(1+32n^{-2}\right) \\
< & 0.2n^{\frac{1}{2}}+7n^{-\frac{3}{2}}
\end{align*}
when $n$ is even. For $n\geq 49$, it is easy to see that
$$ 1.3n^\frac{1}{2}>7.5n^{-\frac{1}{2}}+3^\frac{3}{2}+7n^{-\frac{3}{2}} $$
and hence
$$\frac{3}{2}\left(n^\frac{1}{2}-5n^{-\frac{1}{2}}\right)-3^\frac{3}{2}>0.2n^{\frac{1}{2}}+7n^{-\frac{3}{2}}. $$
Similarly, when $n$ is odd,
\begin{align*}
& l\left(n,\frac{n-1}{2}-1,\frac{n}{2}\right) \\
= & \sqrt{2}\left(\frac{n-1}{2}-1\right)\left(\sqrt{\frac{n-2}{\frac{n-1}{2}(n-\frac{n-1}{2})}}-\sqrt{\frac{\frac{3n-9}{2}}{\frac{(n-3)}{2}(n-1)}}\right) \\
= & \frac{\sqrt{2}}{2}(n-3)\frac{\frac{4(n-2)}{n^{2}-1}-\frac{3}{n-1}}{\sqrt{\frac{4(n-2)}{n^{2}-1}}+\sqrt{\frac{3}{n-1}}}\\
= & \frac{\sqrt{2}}{2}\frac{n^{2}-14n+33}{\sqrt{4(n^3-2n^2-n+2)}+\sqrt{3(n^3+n^2-n-1)}} .
\end{align*}
Then
\begin{align*}
& l\left(n,\frac{n-1}{2}-1,\frac{n}{2}\right) \\
< & \frac{\sqrt{2}}{2}\frac{n^2+33}{\sqrt{4(n^3-3n^2)}+\sqrt{3n^3}}\\
= & \frac{\sqrt{2}}{2}\frac{n^{\frac{1}{2}}+33n^{-\frac{3}{2}}}{\sqrt{4(1-3n^{-1})}+\sqrt{3}} \\
\leq & \frac{\sqrt{2}}{2}\frac{1}{\sqrt{4(1-\frac{3}{49})}+\sqrt{3}}\left(n^{\frac{1}{2}}+33n^{-\frac{3}{2}}\right) \\
= & \frac{\sqrt{2}}{2}\frac{7}{2\sqrt{46}+7\sqrt{3}}\left(n^{\frac{1}{2}}+33n^{-\frac{3}{2}}\right)\\
< & 0.2\left(n^{\frac{1}{2}}+33n^{-\frac{3}{2}}\right).
\end{align*}
For $n\geq 49$ we have
$$ 1.3n^\frac{1}{2}>7.5n^{-\frac{1}{2}}+3^\frac{3}{2}+6.6n^{-\frac{3}{2}} $$
and hence
$$ \frac{3}{2}\left(n^\frac{1}{2}-5n^{-\frac{1}{2}}\right)-3^\frac{3}{2}>0.2\left(n^{\frac{1}{2}}+33n^{-\frac{3}{2}}\right). $$
\end{itemize}
\end{proof}

\section{Some computational analysis}
With Theorems~\ref{theorem4-1}, \ref{theorem5-1} and \ref{theorem6-1}, we may examine the influence on the maximum ABC index by the independence number $\beta$, pendent vertex number $p$, and edge-connectivity number $k$. In Figure~\ref{fig:trendingoftheo2.2} we take $n=200, 250, 300, 350$ respectively and $\beta \in [1,199]$, it is easy to see that the maximum ABC index is decreasing faster as $\beta$ grows.

\begin{figure*}[!htbp]
\centering
\includegraphics[width=0.6\textwidth]{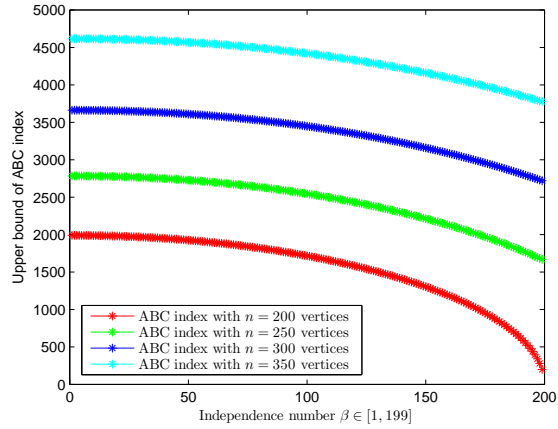}
\caption{The maximum ABC index with $n=200, 250, 300, 350$ and $\beta \in [1,199]$.}
\label{fig:trendingoftheo2.2}
\end{figure*}

Similarly, Figures~\ref{fig:trendingoftheo2.4} and \ref{fig:trendingoftheo3.1} show that the maximum ABC index decreases, but slower as the number of pendant vertices or edge-connectivity grows.

\begin{figure*}[!htbp]
\centering
\includegraphics[width=0.6\textwidth]{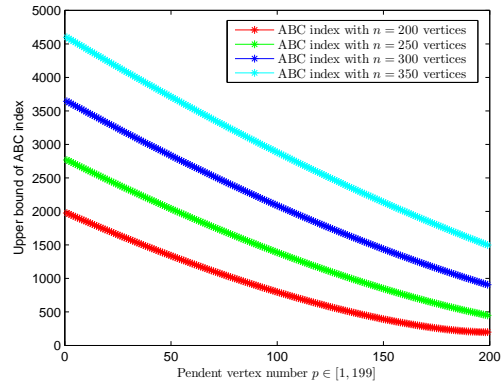}
\caption{The maximum ABC index with $n=200, 250, 300, 350$ and $p \in [1,199]$.}
\label{fig:trendingoftheo2.4}
\end{figure*}

\begin{figure*}[!htbp]
\centering
\includegraphics[width=0.6\textwidth]{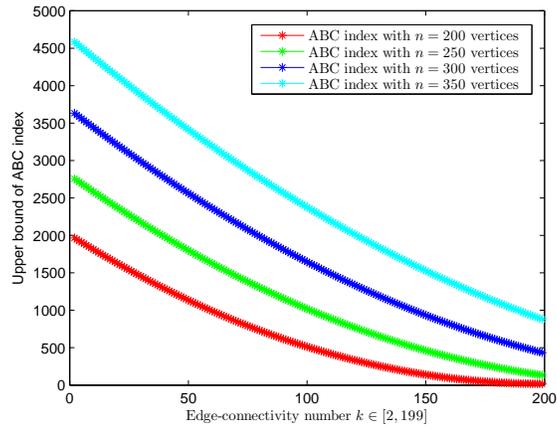}
\caption{The maximum ABC index with $n=200, 250, 300, 350$ and $k \in [2,199]$.}
\label{fig:trendingoftheo3.1}
\end{figure*}

In Figure~\ref{fig:trendingofinter} the curves corresponding to $n=200$, $\beta, p \in [1,199]$, and $k \in [2,199]$ are plotted. It is interesting to note that with given value $x$, the maximum ABC index is the largest when $\beta = x$ and smallest when $k=x$.

\begin{figure*}[!htbp]
\centering
\includegraphics[width=0.6\textwidth]{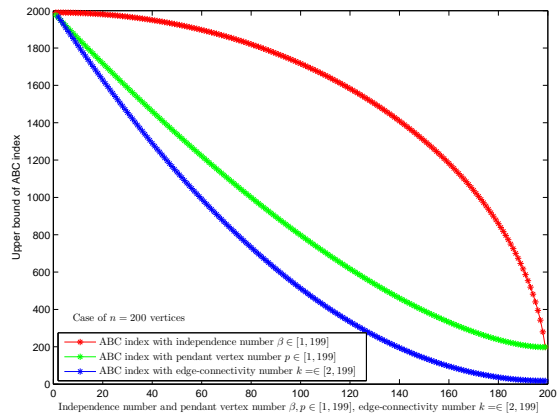}
\caption{The maximum ABC index with $n=200$, $\beta, p \in [1,199]$, and $k \in [2,199]$.}
\label{fig:trendingofinter}
\end{figure*}

\section{Concluding Remarks}

We have discussed the maximum ABC index among graphs of given order and various fixed parameters. As can be seen from the arguments, the ideas are simple but the proofs can be very technical and tedious. As another example of such studies, one may consider the maximum ABC index of graphs with given chromatic number.

\begin{definition}
Denote by $T_{n,t}$ the complete $t-$partite graph of order $n$ with $|n_i-n_j|\leq 1$, where $n_i$, $i=1,2,\cdots, t$, is the number of vertices in the $i$th partition set of  $T_{n,t}$.
\end{definition}

\begin{proposition}\label{theorem7-1}
For any connected graph $G$ of order $n$ with chromatic number $\chi=2$:
\begin{itemize}
\item If $n$ is even, then $ABC(G)\leq \frac{n}{2}\sqrt{n-2}$ with equality if and only if $G\cong T_{n,2}$;
\item If $n$ is odd, then $ABC(G)\leq \frac{1}{2}\sqrt{(n-2)(n^2-1)}$ with equality if and only if $G\cong T_{n,2}$.
\end{itemize}
\end{proposition}
\begin{proof}
Let $G^\ast$ be the graph with the maximum $ABC$ index among all $n$-vertex connected graphs with chromatic number $\chi=2$.
By Theorem \ref{theorem2-2}, we must have $G^\ast \cong \overline{K_{n_1}}\vee \overline{K_{n_2}}$, where $n_i$ is the number of vertices in the $i$th partition set.

Suppose (for contradiction) that $G^\ast \ncong T_{n,2}$ and $n_2\geq n_1+2$, consider $G'=\overline{K_{n_1+1}}\vee \overline{K_{n_2-1}}$ and we have
\begin{align*}
& ABC(G')-ABC(G^\ast) \\
= & (n_1+1)(n_2-1)\sqrt{\frac{2n-n_1-n_2-2}{(n_1+1)(n_2-1)}}-n_1n_2\sqrt{\frac{2n-n_1-n_2-2}{n_1n_2}} \\
= & \left(\sqrt{(n_1+1)(n_2-1)}-\sqrt{n_1n_2}\right)\sqrt{n-2}.
\end{align*}
Since $(n_1+1)(n_2-1)-n_1n_2=n_2-n_1-1>0$, we have $ABC(G')-ABC(G^\ast)>0$, a contradiction.
\end{proof}

Both computational results and combinatorial intuitions suggest the following, which we post here as a conjecture.

\begin{conjecture}\label{conjecture7-2}
Let $G$ be an $n-$vertex connected graph with chromatic number $\chi\geq 3$. Then
$$ABC(G)\leq ABC(T_{n,\chi})$$ with equality if and only if $ABC(G)\cong ABC(T_{n,\chi})$.
\end{conjecture}

Another question is to consider the case when the edge connectivity is 1. We conjecture that the maximum ABC index behaves similarly as in the general case, achieved by attaching a pendant edge to a vertex of $K_{n-1}$. Note that to prove this, it suffices to show that the following function (with $f(x,y)=\sqrt{\frac{x+y-2}{xy}}$) is decreasing:
\begin{align*}
& (x-1)f(x,x-1) + \frac12(x-1)(x-2)f(x-1,x-1) + f(x,n-x)  \\
+ &  (n-x-1)f(n-x,n-x-1) \\
 + & \frac12 (n-x-1)(n-x-2) f(n-x-1,n-x-1) .
\end{align*}
\begin{center}
{\bf Acknowledgements}
\end{center}
  The authors would like to thank the anonymous referees for many helpful and  comments on an earlier version of this paper.

\end{document}